\documentclass[11pt]{article}
\usepackage[utf8]{inputenc} 
\usepackage{latexsym,amsmath,amssymb,textcomp}
\usepackage[all]{xy}
\usepackage[nottoc, notlof, notlot]{tocbibind}
\usepackage{geometry} 
\usepackage{graphicx} 
\usepackage{mathtools}
\usepackage{amsthm}
\usepackage{amsmath}
\usepackage{booktabs}
\usepackage{array} 
\usepackage{paralist}
\usepackage{verbatim} 
\usepackage{subfig} 
\usepackage{ stmaryrd }
\usepackage{hyperref}
\usepackage{fancyhdr} 
\title{Equivalences between blocks of $p$-local Mackey algebras.}
\author{Baptiste Rognerud}
\begin{document}
\maketitle
\theoremstyle{plain}
\newtheorem{theo}{Theorem}[section]
\newtheorem{theox}{Theorem}[section]
\renewcommand{\thetheox}{\Alph{theox}}
\newtheorem{prop}[theo]{Proposition}
\newtheorem{lemma}[theo]{Lemma}
\newtheorem{coro}[theo]{Corollary}
\newtheorem{question}[theo]{Question}
\newtheorem{notations}[theo]{Notations}
\theoremstyle{definition}
\newtheorem{de}[theo]{Definition}
\newtheorem{dex}[theox]{Definition}
\theoremstyle{remark}
\newtheorem{ex}[theo]{Example}
\newtheorem{re}[theo]{Remark}
\newtheorem{res}[theo]{Remarks}
\renewcommand{\labelitemi}{$\bullet$}
\newcommand{\comu}{co\mu}
\newcommand{\mo}{\ensuremath{\hbox{mod}}}
\newcommand{\Mo}{\ensuremath{\hbox{Mod}}}
\newcommand{\decale}[1]{\raisebox{-2ex}{$#1$}}
\newcommand{\decaleb}[2]{\raisebox{#1}{$#2$}}
\begin{abstract}Let $G$ be a finite group and $(K,\mathcal{O},k)$ be a $p$-modular system. Let $R=\mathcal{O}$ or $k$. There is a bijection between the blocks of the group algebra and the blocks of the so-called $p$-local Mackey algebra $\mu_{R}^{1}(G)$. Let $b$ be a block of $RG$ with abelian defect group $D$. Let $b'$ be its Brauer correspondant in $N_{G}(D)$. It is conjectured by Brou\'e that the blocks $RGb$ and $RN_{G}(D)b'$ are derived equivalent. Here we look at equivalences between the corresponding blocks of $p$-local Mackey algebras. We prove that an analogue of the Brou\'e's conjecture is true for the $p$-local Mackey algebras in the following cases: for the principal blocks of $p$-nilpotent groups and for blocks with defect $1$. We also point out the probable importance of \emph{splendid} equivalences for the Mackey algebras. 
\end{abstract}
\par\noindent
{\it {\footnotesize A.M.S. subject classification: 20C05, 18E30,16G10,16W99.}}

\section{Introduction and preliminaries}
\subsection{Introduction}
Let $G$ be a finite group and $(K,\mathcal{O},k)$ be a $p$-modular system. Let $R=\mathcal{O}$ or $k$. In their paper, Th\'evenaz and Webb proved that there is a bijection $b\mapsto b^{\mu}$ between the blocks of $RG$ and the primitive central idempotents of $\mu_{R}^{1}(G)$, called the blocks of $\mu_{R}^{1}(G)$, where $\mu_{R}^{1}(G)$ is the so called $p$-local Mackey algebra.  
\newline This bijection preserves the defect groups (which are the same as the defect groups of the corresponding block of $RG$, see Section $5.3$ of \cite{bouc_blocks}). So using the Brauer correspondence, we have the following diagram: Let $b$ be a block of $RG$ with defect group $D$ and $b'$ be the Brauer correspondent of $b$ in $RN_{G}(D)$. 
\begin{equation*}
\xymatrix{
b\in Z(RG)\ar[r]\ar[d] & b^{\mu}\in Z(\mu_{R}^{1}(G)\ar[d]) \\
b'\in Z(RN_{G}(D))\ar[r] & b'^{\mu}\in Z(\mu_{R}^{1}(N_{G}(D))).
}
\end{equation*}
If $D$ is abelian, it is conjectured by Brou\'e that the block algebras $RGb$ and $RN_{G}(D)b'$ are derived equivalent. In \cite{equiv_coho}, the Author proved that if two blocks of group algebras are \emph{splendidly} derived equivalent, then the corresponding blocks of the so-called \emph{cohomological} Mackey algebras are derived equivalent. The cohomological Mackey algebra is a quotient of the $p$-local Mackey algebra. So, it is a very natural question to ask if the same can happen for the corresponding $p$-local Mackey algebras which contain much more informations than their cohomological quotient (see Proposition \ref{decomp} e.g.). 
\begin{question}[Bouc]\label{q1}
Let $G$ be a finite group and $b$ be a block of $\mathcal{O}G$ with abelian defect group $D$. Let $b'$ be its Brauer correspondent in $\mathcal{O}N_{G}(D)$. Is there a derived equivalence $D^{b}(\mu_{\mathcal{O}}^{1}(G)b^{\mu})\cong D^{b}(\mu_{\mathcal{O}}^{1}(N_{G}(D))b'^{\mu})$ ?
\end{question}
 Unfortunately, the tools which was developed for the cohomological Mackey algebras cannot be used here. Moreover Question \ref{q1} is much harder than its analogue for the cohomological Mackey algebras and is connected to deep questions of representation of finite groups such as the knowledge of the indecomposable $p$-permutation modules (see Section $6$ and $7$ of \cite{bouc_resolution}). Here, we will not answer this question in general, but we consider it in the following two cases: first for the Mackey algebra of the principal block of $p$-nilpotent groups, and then for groups with Sylow $p$-subgroups of order $p$. 
\newline The main result of this paper is the following theorem:
\begin{theo}
The answer to Question \ref{q1} is affirmative in the following cases:
\begin{itemize}
\item For principal blocks of $p$-nilpotent groups. Here we do not assume the Sylow $p$-subgroups to be abelian. Moreover it is a Morita equivalence.
\item When the $p$-local Mackey algebra is symmetric. That is for the blocks with defect $1$. 
\end{itemize}
\end{theo}
For a non principal block of a $p$-nilpotent group, or more generally for a nilpotent block it is not expect for the corresponding block of the Mackey algebra to be Morita equivalent to the Mackey algebra of its defect group. 
\newline Let $b$ and $c$ be two blocks of group algebras. Then it is probable that the corresponding blocks of $p$-local Mackey algebras will be equivalent only if there is an equivalence between $b$ and $c$ which respect the $p$-permutation modules (\emph{splendid} equivalences e.g.). We do not have a proof of this fact, but there are some clues in this direction. Example \ref{ex_sl} is exactly an example of blocks which are Morita equivalent but not `splendidly' Morita equivalent. In this case, the corresponding blocks of $p$-local Mackey algebras are not Morita equivalent but derived equivalent. Other example can be found in Chapter $4$ of \cite{these}. Moreover the decomposition matrix of a block of $p$-local Mackey algebras involves $p$-permutation modules and Brauer quotient (see Proposition \ref{decomp}).
\newline In the first part, we recall some basic definitions and results on Mackey functors. Since we will use several points of view on Mackey functors, we give a full proof of the well-known fact that all these points of view are equivalent. Then we recall the bijection between the blocks of the group algebra and the blocks of the $p$-local Mackey algebra. At the end of the second section, we explain how the decomposition matrix of the $p$-local Mackey algebra of a finite group $G$ can be computed from the knowledge of the ordinary characters of some $p$-local subgroups of $G$ and the indecomposable $p$-permutation $kG$-modules.
\newline In the last part, we investigate on some basic properties of equivalences between blocks of $p$-local Mackey algebras. We show that naive thoughts about these questions are seldom true. For example an equivalence between blocks of $p$-local Mackey algebras do not need to induce an equivalence for the cohomological quotient. If two blocks of group algebras are isomorphic (resp. Morita equivalent), then the corresponding blocks of $p$-local Mackey algebras do not need to be isomorphic (resp. Mortia equivalent).
\newline We show that there is a Morita equivalence between the principal block of the $p$-local Mackey algebra of a $p$-nilpotent group $G$ and the Mackey algebra of a Sylow $p$-subgroup of $G$. Moreover, we determine the structure of algebra of the block of the $p$-local Mackey algebra in this case. This can be view as an analogue of Puig's theorem (see \cite{nilpotent_puig} or \cite{nilpotent_revisited}) for the Mackey algebras. 
\newline Then we will answer Question $1.1$ for blocks with defect group of order $p$. The result uses the fact that the Mackey algebras are \emph{symmetric} algebras in this situation. More precisely over the field $k$, they are \emph{Brauer tree} algebras and over the valuation ring $\mathcal{O}$ they are \emph{Green orders}.\newline It should be notice that all the equivalences between blocks of groups algebras which induce equivalences between the corresponding $p$-local Mackey algebras that we were able to construct are in fact \emph{splendid}. 
\paragraph{Notation:} Let $R$ be a commutative ring with unit. We denote by $R$-$Mod$ the category of (all) $R$-modules.
\newline Let $G$ be a finite group. Let $p$ be a prime number. We denote by $(K,\mathcal{O},k)$ a $p$-modular system, i-e $\mathcal{O}$ is a complete discrete valuation ring with maximal ideal $\mathfrak{p}$, such that $\mathcal{O}/\mathfrak{p}=k$ is a field of characteristic $p$ and $Frac(\mathcal{O})=K$ is a field of characteristic zero. If $R=\mathcal{O}$ or $k$, then the direct summands of the permutation $RG$-modules are called $p$-\emph{permutation} modules. We denote by $G$-$set$ the category of finite $G$-sets. If $H$ is a subgroup of $G$ then, we denote by $N_{G}(H)$ its normalizer in $G$. The quotient $N_{G}(H)/H$ will be, sometimes, denoted by $\overline{N}_{G}(H)$. We denote by $\Omega_{G}$ the disjoint union of all transitive $G$-sets. 
\newline Let $V$ be an $RG$-module, for $R=\mathcal{O}$ or $k$. Let $Q$ be a $p$-subgroup of $G$. Since we are dealing with Mackey functor we use the notation $V[Q]$ instead of the usual $V(Q)$ for the Brauer quotient of $V$ (see Section $1$ \cite{broue_scott}).
\newline If $F : \mathcal{A}\to \mathcal{B}$ and $G: \mathcal{B}\to \mathcal{A}$ are two functors, we denote by $F \dashv G$ the fact that $F$ is a left adjoint of $G$.\\  
 N.B. We will denote by the same letter the block idempotents for the ring $\mathcal{O}$ and the field $k$.  
\section{Basic results on Mackey functors.}
Let $G$ be a finite group, and $R$ be a commutative ring. There are several definitions of Mackey functors for $G$ over a ring $R$, the first one was introduced by Green in \cite{green}: 
\begin{de}
A Mackey functor for $G$ over $R$ consists of the following data: 
\begin{itemize}
\item For every subgroup $H$ of $G$, an $R$-module denoted by $M(H)$. 
\item For subgroups $H\subseteq K$ of $G$, a morphism of $R$-modules \newline $t_{H}^{K}:M(H)\rightarrow~M(K)$ called transfer, or induction, and a morphism of $R$-modules $r^{K}_{H}: M(K)\rightarrow M(H)$ called restriction.
\item For every subgroup $H$ of $G$, and each element $x$ of $G$, a morphism of $R$-modules $c_{x,H}: M(H)\rightarrow M(^{x}H)$ called conjugacy map.
\end{itemize}
such that:
\begin{enumerate}
\item \emph{Triviality axiom}: For each subgroup $H$ of $G$, and each element $h\in H$, the morphisms $r^{H}_{H}$, $t^{H}_{H}$ et $c_{h,H}$ are the identity morphism of $M(H)$.                         
\item \emph{Transitivity axiom}: If $H \subseteq K \subseteq L$ are subgroups of $G$, then 
$t^{L}_{K}\circ t^{K}_{H} = t^{L}_{H}$ and $r^{K}_{H}\circ r^{L}_{K} = r^{L}_{H}$. Moreover if $x$ and $y$ are elements of $G$, then $c_{y,^{x}H}\circ c_{x,H}=c_{yx,H}$.
\item \emph{Compatibility axioms:} If $H\subseteq K$ are subgroups of $G$, and if $x$ is an element of $G$, then $c_{x,K}\circ t^{K}_{H} = t^{^{x}K}_{^{x}H}\circ c_{x,H}$ et $c_{x,H}\circ r^{K}_{H} = r^{^{x}K}_{^{x}H}\circ c_{x,K}$.
\item \emph{Mackey axiom:} If $H\subseteq K \supseteq L$ are subgroups of $G$, then 
\begin{equation*}
r^{K}_{H}\circ t^{K}_{L} = \sum_{x\in [H\backslash K / L]} t^{H}_{H\cap ^{x}L}\circ c_{x,H^{x}\cap L}\circ r^{L}_{H^{x} \cap L}.
\end{equation*} 
where $[H\backslash K / L]$ is a set of representatives of the double cosets $H\backslash K /L$. 
\end{enumerate}
\end{de}
\noindent In particular, for each subgroup $H$ of $G$, the $R$-module $M(H)$ is an {$N_{G}(H)/H$-module}. \\
A morphism $f$ between two Mackey functors $M$ and $N$ is the data of a $R$-linear morphism $f(H) : M(H)\to N(H)$ for every subgroup $H$ of $G$. These morphisms are compatible with transfer, restriction and conjugacy maps. We denote by $Mack_{R}(G)$ the category of Mackey functors for $G$ over $R$.
\begin{ex} 
Let $V$ be an $RG$-module. The fixed point functor $FP_{V}$ is the Mackey functor for $G$ over $R$ defined as follows:\newline For $H\leqslant G$, then $FP_{V}(H)=V^{H}:=\{\ v\in V\ ; \ h\cdot v=v \ \forall\ h\in H\ \}$. If $H\leqslant K\leqslant G$, we have $V^{K}\subseteq V^{H}$, so we define the restriction map $r^{K}_{H}$ as the inclusion map. The transfer map $ t^{K}_{H}: V^{H}\to V^{K}$ is defined by $t_{H}^{K}(v)=\sum_{k\in [K/H]}k\cdot v$ where $[K/H]$ is a set of representatives of $K/H$. The conjugacy maps are induced by the action of $G$ on $V$. 
\end{ex}
It is not hard to see that the construction $V\mapsto FP_{V}$ is a functor from $RG$-$Mod$ to $Mack_{R}(G)$. \newline Conversely we have an obvious functor $ev_{1}: Mack_{R}(G)\to RG$-$Mod$ given by the evaluation at the subgroup $\{1\}$. 
\begin{prop}\cite{tw}
The functor $ev_{1}$ is a left adjoint to the functor $V\mapsto FP_{V}$. I.e. there is a natural isomorphism: 
\begin{equation*}
Hom_{Mack_{R}(G)}(M,FP_{V})\cong Hom_{RG}(M(1),V)
\end{equation*}
for a Mackey functor $M$ and an $RG$-module $V$. 
\end{prop}
\begin{re}
The full subcategory of $Mack_{R}(G)$ consisting of fixed point functors $FP_{V}$ is equivalent to the category $RG$-$Mod$ (see proof of Theorem 18.1 \cite{tw}). 
\end{re}
An other definition of Mackey functors was given by Dress in \cite{dress}: 
\begin{de}\label{def_dress}
A bivariant functor $M~=~(M^{*},M_{*})$ from $G$-$set$ to $R$-$Mod$ is a pair of functors from $G$-$set$ to $R$-$Mod$ such that $M^{*}$ is a contravariant functor, and $M_{*}$ is a covariant functor. If $X$ is a finite $G$-set, then the image by the covariant and by the contravariant part coincide. We denote by $M(X)$ this image. A Mackey functor for $G$ over $R$ is a bivariant functor from $G$-$set$ to $R$-$Mod$ such that: 
\begin{itemize}
\item Let $X$ and $Y$ be two finite $G$-sets, $i_{X}$ and $i_{Y}$ the canonical injection of  $X$ (resp. $Y$) in $X\sqcup Y$. Then $M^{*}(i_X)\oplus M^{*}(i_Y)$ and $(M_{*}(i_X),M_{*}(i_Y))$ are inverse isomorphisms, from $M(X)\oplus M(Y)$ to $M(X\sqcup Y)$.
\item If
\begin{equation*}
\xymatrix{
X\ar[r]^{a}\ar[d]^{b}& Y\ar[d]^{c} \\
Z\ar[r]^{d} & T 
}
\end{equation*}
is a pull back diagram of $G$-sets, then the diagram 
\begin{equation*}
\xymatrix{
M(X)\ar[d]_{M_{*}(b)} & M(Y)\ar[l]_{M^{*}(a)}\ar[d]^{M_{*}(c)}\\
M(Z) & M(T)\ar[l]_{M^{*}(d)}
}
\end{equation*}
is commutative. 
\end{itemize}
A morphism between two Mackey functors is a natural transformation of bivariant functors. 
\begin{ex}[2.4.1 \cite{bouc_green}]\label{burnside}
If $X$ is a finite $G$-set, the category of $G$-sets over $X$ is the category with objects $(Y,\phi)$ where $Y$ is a finite $G$-set and $\phi$ is a morphism from $Y$ to $X$. A morphism $f$ from $(Y,\phi)$ to $(Z,\psi)$ is a morphism of $G$-sets $f:Y\to Z$ such that $\psi\circ f=\phi$. 
\newline\indent The Burnside functor at $X$, written $B(X)$, is the Grothendieck group  of the category of $G$-sets over $X$, for relations given by disjoint union. This is a Mackey functor for $G$ over $R$ by extending scalars from $\mathbb{Z}$ to $R$. We will denote $RB$ the Burnside functor after scalars extension. 
\newline If $X$ is a $G$-set, the Burnside group $RB(X^2)$ has a ring structure. A $G$-set $Z$ over $X\times X$ is the data of a $G$-set $Z$ and a map $(b\times a)$ from $Z$ to $X\times X$, denoted by $(X\overset{b}{\leftarrow} Y \overset{a}{\rightarrow} X)$. The product of (the isomorphism class of ) $(X\overset{\alpha}{\leftarrow} Y \overset{\beta}{\rightarrow} X)$ and (the isomorphism class of )$(X\overset{\gamma}{\leftarrow} Z \overset{\delta}{\rightarrow} X)$ is given by (the isomorphism class of) the pullback along $\beta$ and $\gamma$. 
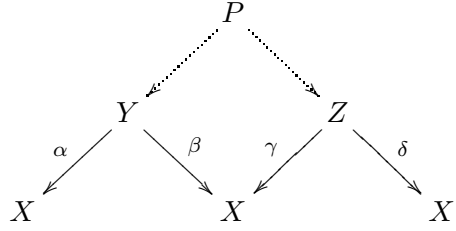
\begin{equation*}
\xymatrix{
& & P\ar@{..>}[dr]\ar@{..>}[dl] & & \\
& Y\ar[dl]_{\alpha}\ar[dr]^{\beta} & & Z\ar[dl]_{\gamma}\ar[dr]^{\delta} & \\
X & & X & &X
}
\end{equation*}
The identity of this ring is (the isomorphism class) $
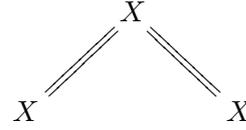
\xymatrix{&X\ar@{=}[rd]\ar@{=}[dl]&\\X& &X }$
\newline In the rest of the paper, we will use the same notation for a $G$-set over $X\times X$ and its isomorphism class in $RB(X\times X)$.
\end{ex}
\end{de}
We need a last definition of Mackey functors which was given by Th\'evenaz and Webb in \cite{tw}, and uses the Mackey algebra. 
\begin{de}
The Mackey  algebra $\mu_{R}(G)$ for $G$ over $R$ is the unital associative algebra with generators  $t_{H}^{K}$, $r^{K}_{H}$ and $c_{g,H}$ for $H\leqslant K\leqslant G$ and $g\in G$, with the following relations:
\begin{itemize}
\item $\sum_{H\leqslant G}t^{H}_{H}=1_{\mu_{R}(G)}$. 
\item $t^{H}_{H}=r^{H}_{H}=c_{h,H}$ for $H\leqslant G$ and $h\in H$. 
\item $t^{L}_{K}t_{H}^{K}=t^{L}_{H}$, $r^{K}_{H}r^{L}_{K}=r^{L}_{H}$ for $H\subseteq K\subseteq L$. 
\item $c_{g',{^{g}H}}c_{g,H}=c_{g'g,H}$, for $H\leqslant G$ and $g,g'\in G$. 
\item $t^{{^{g}K}}_{{^{g}H}}c_{g,H}=c_{g,K}t^{K}_{H}$ and $r^{{^{g}K}}_{{^{g}H}}c_{g,K}=c_{g,H}r^{K}_{H}$, $H\leqslant K$, $g\in G$. 
\item $r^{H}_{L}t^{H}_{K}=\sum_{h\in [L\backslash H / K]} t^{L}_{L\cap {^{h} K}} c_{h, L^{h} \cap K} r^{K}_{L^{h}\cap K}$ for $L\leqslant H \geqslant K$. 
\item All the other products of generators are zero. 
\end{itemize}
\end{de}
\begin{de}
A Mackey functor for $G$ over $R$ is a left $\mu_{R}(G)$-module.
\end{de}
\begin{prop}\label{basis}
The Mackey algebra is a free $R$-module, of finite rank independent of $R$. The set of elements $t^{H}_{K}xr^{L}_{K^{x}}$, where $H$ and $L$ are subgroups of $G$, where $x\in [H\backslash G/L]$, and $K$ is a subgroup of $H{\cap~{\ ^{x}L}}$ up to $(H\cap {\ ^{x}L})$-conjugacy, is an $R$-basis of $\mu_{R}(G)$.  
\end{prop}
\begin{proof}
Section $3$ of \cite{tw}. 
\end{proof}
Let us recall that the Mackey algebra is isomorphic to a Burnside algebra:
\begin{prop}[Proposition 4.5.1 \cite{bouc_green}]\label{burnside_alg}
The Mackey algebra $\mu_{R}(G)$ is isomorphic to $RB(\Omega_{G}^{2})$, where $\Omega_{G}=\sqcup_{L\leqslant G} G/L$. 
\end{prop}
\begin{proof}
Let $H\leqslant K$ be two subgroups of $G$, then we denote by $\pi^{K}_{H}$ the natural surjection from $G/H$ to $G/K$. If $g\in G$, then we denote by $\gamma_{H,g}$ the map from $G/\ ^{g}H$ to $G/H$ defined by $\gamma_{H,g}(xgHg^{-1}) =xgH$. The isomorphism $\beta$ is defined on the generators of $\mu_{R}(G)$ by: $$\beta(t_{H}^{K}) = \xymatrix{ & G/H\ar[dl]_{\pi^{K}_{H}}\ar@{=}[rd]  & \\ \Omega_{G}\supset G/K&& G/H\subset \Omega_{G}}$$
$$\beta(r_{H}^{K}) = \xymatrix{ & G/H\ar[dr]^{\pi^{K}_{H}}\ar@{=}[ld]  & \\ \Omega_{G}\supset G/H&& G/K\subset \Omega_{G}}$$
$$\beta(c_{g,H}) =  \xymatrix{ & G/\ ^gH\ar[dr]^{\gamma_{H,g}}\ar@{=}[ld]  & \\ \Omega_{G}\supset G/\ ^gH&& G/H\subset \Omega_{G}}$$
\end{proof}
We will make an intensive use of the connection between the different categories of Mackey functors, so let us recall the following well known result:
\begin{prop}\cite{tw}
The different definitions of Mackey functors for $G$ over $R$ are equivalent. 
\end{prop}
\begin{proof}[Sketch of proof] The idea of the proof can be found in various places such as \cite{tw}, or \cite{bouc_green}. Here we propose a quite different, and more conceptual  approach where we do not need to choose representatives. For this proof, we denote by $Mack_{R}^{(a)}(G)$ the category of Mackey functors for $G$ over $R$ in the sense of Green and by $Mack_{R}^{(b)}(G)$ the category of Mackey functors for Dress. The most cost-effective way seems to prove the following equivalences:
\begin{equation*}
\xymatrix{
Mack_{R}^{(a)}(G)\ar[dr]^{\cong} &&& Mack_{R}^{(b)}(G)\ar[dl]_{\cong} \\
& \mu_{R}(G)\hbox{-}Mod\ar[r]^{\cong} &  RB(\Omega_{G}^2)\hbox{-}Mod
}
\end{equation*}
For the equivalence $Mack_{R}^{a}(G) \cong \mu_{R}(G)$-$Mod$, if $M\in Mack_{R}^{(a)}(G)$, then $\bigoplus_{H\leqslant G} M(H)$ is a $\mu_{R}(G)$-module. Conversely, if $V$ is a $\mu_{R}(G)$-module, then one can define a Mackey functor $M$ by: for $H\leqslant G$, then $M(H):=t^{H}_{H}V$. The Mackey functor structure of $M$ is induced by the module structure of $V$.
\newline The equivalence between $\mu_{R}(G)$-$Mod$ and $RB(\Omega_{G}^{2})$-$Mod$ follows from the isomorphism of Proposition \ref{burnside_alg}. 
\newline The equivalence $Mack_{R}^{(b)}(G)\cong RB(\Omega_{G}^{2})$-$Mod$, is a bit more technical. Let $M$ be a Mackey functor for $G$ over $R$, using Dress definition. Then $M(\Omega_{G})$ is a $RB(\Omega_{G}^{2})$-module. The action of a $G$-set $Z=(\Omega_{G}\overset{b}{\leftarrow} Z\overset{a}{\rightarrow} \Omega_{G} )$ over $\Omega_{G}\times \Omega_{G}$ on $M(\Omega_{G})$ is defined by: let $m\in M(\Omega_{G})$, 
\begin{equation*}
Z\cdot m= M_{*}(b)\circ M^{*}(a)(m).
\end{equation*}
Let $V$ be an $RB(\Omega_{G}^{2})$-module, let $X$ be a finite $G$-set. Then $RB_{X} =RB(\Omega_{G}\times X)$ is an $RB(\Omega_{G}^{2})$-module. The action of $RB(\Omega_{G}^{2})$ on $RB_{X}$ is given by left multiplication. Let $f : X\to Y$ be a morphism between two finite $G$-sets. Let us denote by $\hat{f}$ the morphism $Id_{\Omega_{G}}\times f$. Let $\phi : Z\to X\times \Omega_{G}$ be a $G$-set over $X\times \Omega_{G}$ and let $\psi: W\to Y\times \Omega_{G}$ be a $G$-set over $Y\times \Omega_{G}$. Then there is a morphism of $RB(\Omega_{G}^{2})$-modules: $RB(f)_{*} : RB_{X}\to RB_{Y}$ defined by:
\begin{equation*}
RB_{*}(f)(\phi)=\phi\circ \hat{f}.
\end{equation*}
On the other direction, there is a morphism of $RB(\Omega_{G}^{2})$-modules defined by: $$RB^{*}(f)(\psi)=\pi,$$ where $\pi : P\to X\times \Omega_{G}$ is a $G$-set over $X\times \Omega_{G}$ such that the following diagram is a pullback diagram:
\begin{equation*}
\xymatrix{
P\ar[r]\ar[d]^{\pi} & W\ar[d]^{\psi}\\
X\times \Omega_{G} \ar[r]^{\hat{f}}& Y\times \Omega_{G}
}
\end{equation*}
then one can defined a Mackey functor $Y_{V}$ for $G$ over $R$ by:
\begin{equation*}
Y_{V}(X) = Hom_{RB(\Omega_{G}^{2})}(RB_{X},V). 
\end{equation*}
The Mackey functor structure of $Y_{V}$ is induced by the  covariant and the contravariant part of $RB$. 
\newline The functor from $Mack_{R}^{(b)}(G)$ to $RB(\Omega_{G}^{2})$-$Mod$ sending $M$ to $M(\Omega_{G})$ is fully faithful, indeed, if $X$ is a finite $G$-set, then let $X=\bigsqcup_{H\leqslant G} n_{H}G/H$ be the decomposition of $X$ as sum of transitive $G$-sets. If a transitive $G$-set $G/H$ appears with a positive coefficient in $X$, then we denote by $i_{H}^{X}$ the canonical injection from $G/H$ to $X$. Moreover, we denote by $i_{G/H}$ the canonical injection from $G/H$ to $\Omega_{G}$. Let $\phi : M\to N$ be a morphism of Mackey functors, then since $\phi$ is a natural transformation of bivariant functors, and since $\sum_{H}n_{H}M_{*}(i_{H}^{X})M^{*}(i_{H}^{X})=Id_{M(X)}$, we have:
\begin{align*}
\phi_{X} &= \sum_{H\leqslant G} n_{H} N_{*}(i_{H}^{X}) \phi_{G/H} M^{*}(i_{H}^{X})\\
&= \sum_{H\leqslant G} n_{H} N_{*}(i_{H}^{X})N^{*}(i_{G/H}) \phi_{\Omega_{G}} M_{*}(i_{G/H})M^{*}(i_{H}^{X}).
\end{align*}
So $\phi$ is characterized by its evaluation at $\Omega_{G}$. Conversely with this formula, if we have $f:M(\Omega_{G})\to N(\Omega_{G})$ a morphism of $RB(\Omega_{G}^{2})$-modules, then one can defined a morphism of Mackey functors $\phi : M\to N$ such that $\phi_{\Omega_{G}}=f$. The fact that this construction gives a natural transformation of bivariant functors follows from the fact that $f$ is a morphism of $RB(\Omega_{G}^{2})$-modules. 
\newline Moreover the functor $M\mapsto M(\Omega_{G})$ is dense. Indeed let $V$ be an $RB(\Omega_{G}^{2})$-module, then
\begin{align*}
Y_{V}(\Omega_{G}) &= Hom_{RB(\Omega_{G}^{2})}(RB_{\Omega_{G}},V)\\
&=Hom_{RB(\Omega_{G}^{2})}(RB(\Omega_{G}^{2}),V)\\
&\cong V. 
\end{align*}
Now, this isomorphism is natural in $V$. So, the categories $Mack_{R}^{(b)}(G)$ and $RB(\Omega_{G}^{2})$-$Mod$ are equivalents, and since $Y_{V}(\Omega_{G})\cong V$, by unicity of the adjoint, the functor $V\mapsto Y_{V}$ is a quasi-inverse equivalence to the functor $M\mapsto M(\Omega_{G})$. 
\end{proof}
In the rest of the paper, if no confusion is possible, we denote by $Mack_{R}(G)$ the category of Mackey functors for $G$ over $R$ for one of these three definitions. 
\section{Blocks of Mackey algebras}
Let $G$ be a finite group and $(K,\mathcal{O},k)$ be a $p$-modular system for $G$ which is ``large enough" for all the $N_{G}(H)/H$ for $H\leqslant G$. In \cite{tw}, Th\'evenaz and Webb proved that there is a bijection between the blocks of the group algebra $\mathcal{O}G$ and the blocks of $Mack_{\mathcal{O}}(G,1)$, where $Mack_{\mathcal{O}}(G,1)$ is the full subcategory of $Mack_{\mathcal{O}}(G)$ consisting of Mackey functors which are projective relatively to the $p$-subgroups of $G$. 
\newline\indent The category $Mack_{\mathcal{O}}(G,1)$ is equivalent to the category of $\mu_{\mathcal{O}}^{1}(G)$-modules where $\mu_{\mathcal{O}}^{1}(G)$ is the subalgebra of $\mu_{\mathcal{O}}(G)$ generated by the $r^{H}_{Q}$, $t^{H}_{Q}$, $c_{Q,x}$ where $Q\leqslant H\leqslant G$, $x\in G$ and $Q$ is a $p$-group. This subalgebra is called the $p$-local Mackey algebra of $G$ over $\mathcal{O}$. The same definitions hold for $K$ or $k$. 
\begin{theo}\label{block}
The set of central primitive idempotents of the $p$-local Mackey algebra $\mu_{\mathcal{O}}^{1}(G)$ is in bijection with the set of the blocks of $\mathcal{O}G$. 
\end{theo}
\begin{re}\label{block_re}
The bijection is moreover explicit. If $b$ is a block of $\mathcal{O}G$ then Bouc gave an explicit formula for the corresponding central idempotent of $\mu_{\mathcal{O}}^{1}(G)$ denoted by $b^{\mu}$ (Theorem 4.5.2 \cite{bouc_blocks}). 
\end{re}
Using the equivalence of categories $Mack_{\mathcal{O}}(G,1)\cong \mu_{\mathcal{O}}^{1}(G)$-$Mod$, we have a decomposition of $Mack_{\mathcal{O}}(G,1)$ into a product of categories, which were called the blocks of $Mack_{\mathcal{O}}(G,1)$ by Th\'evenaz and Webb in \cite{tw} Section $17$. The formula of the idempotents of $\mu_{R}^{1}(G)$ (Remark \ref{block_re}) is rather technical but it is immediate to check that the action of a block idempotent $b^{\mu}$ on the evaluation at $\{1\}$ of a Mackey functor $M$ (using Green's notation for evaluation) is given as follows: let $m\in M(1)$, writing the idempotent $b=\sum_{x\in G}b(x)x$ where $b(x)\in\mathcal{O}$, we have
\begin{equation}
b^{\mu}\cdot m=\sum_{x\in G} b(x)c_{x,1}(m).
\end{equation}
\begin{prop}
Let $R=k$ or $\mathcal{O}$. The set of isomorphism classes of projective indecomposable Mackey functors in a block $b^{\mu}$ of $Mack_{R}(G,1)$ is in bijection with the set of isomorphism classes of indecomposable $p$-permutation modules contained in the block $RGb$.
\end{prop}
\begin{proof}
We use Green's notation for the evaluation of Mackey functors. By Corollary 12.8 of \cite{tw}, we know that the projective indecomposable Mackey functors of $Mack_{k}(G,1)$ are in bijection with the indecomposable $p$-permut\-ation $kG$-modules: if $P$ is an indecomposable projective Mackey functor, then $P(1)$ is an indecomposable $p$-permutation module. Let $Q$ be an other indecomposable projective Mackey functor. Then $P\cong Q$ if and only if $P(1)\cong Q(1)$. The same holds for the projective Mackey functors of $Mack_{\mathcal{O}}(G,1)$. A projective indecomposable Mackey functor $P$ is in the block $b^{\mu}$ if and only if $b^{\mu}\cdot P\neq 0$, but $b^{\mu}\cdot P$ is projective so: 
\begin{align*}
b^{\mu}\cdot P\neq 0 &\Leftrightarrow (b^{\mu}\cdot P)(1)\neq 0 \\
& \Leftrightarrow b\cdot (P(1))\neq 0 \\
&\Leftrightarrow P(1) \hbox{  is in the block $b$ of $RG$.}
\end{align*}
\end{proof}
We will use the following notation: $Mack_{R}(b)$ (resp. $\mu_{R}^{1}(b)$) for the category of Mackey functors which belong to the block $b^{\mu}$ (resp. the algebra $b^{\mu}\mu_{R}^{1}(G)$) for $R= \mathcal{O}$ or $k$. 
\subsection{Brauer construction for Mackey functors and decomposition matrices.}
If $H$ is a subgroup of $G$, then there is an induction functor from $Mack_{R}(H)$ to $Mack_{R}(G)$ denoted by $Ind_{H}^{G}$. There is also a restriction functor from $Mack_{R}(G)$ to $Mack_{R}(H)$ denoted by $Res^{G}_{H}$.
\newline If $N$ is a normal subgroup of $G$, there is an inflation functor $Inf_{G/N}^{G}$ from $Mack_{R}(G/N)$ to $Mack_{R}(G)$. For the definition of these three functors using Green's point of view see Section $4$ and Section $5$ of \cite{tw_simple}. These three functors, using the Dress point of view, arrive from a more general setting, using the fact that a $G$-$H$-biset produces a functor, by `pre-composition', from $Mack_{R}(H)$ to $Mack_{R}(G)$ (see Chapter $8$ \cite{bouc_green} for more details). 
\newline Let $D$ be a finite $G$-set. The Dress construction (see \cite{dress} or $1.2$ \cite{bouc_green}) at $D$ is an endo-functor of the Mackey functors category given by the pre-composition by the cartesian product with $D$: if $M$ is a Mackey functor for $G$ in the sense of Dress, and if $X$ is a finite $G$-set, then the Dress construction of $M$, denoted by $M_{D}$ is:
\begin{equation*}
M_{D}(X)=M(X\times D).
\end{equation*}
Let $R$ be a commutative ring. Let $Q$ be a $p$-subgroup of $G$. The Brauer construction for Mackey functors is a functor $Mack_{R}(G)\to Mack_{R}(\overline{N}_{G}(Q))$ denoted by $M\mapsto M^{Q}$. If $M\in Mack_{R}(G)$, then for $N/Q$ a subgroup of $\overline{N}_{G}(Q)$, 
\begin{equation*}
M^{Q}(N/Q)=M(N)/\sum_{Q\nless R<N} t^{N}_{R}(M(R)).
\end{equation*}
For more details about this functor and for the Mackey functor structure of $M^{Q}$ see Section $5$ \cite{tw_simple}, where $M^{Q}$ is denoted by $M^{+}$. This functor generalizes the Brauer construction for modules since the evaluation at the subgroup $\{1\}$ of $\overline{N}_{G}(Q)$ is 
\begin{equation*}
M^{Q}(Q/Q)=\overline{M}(Q):=M(Q)/\sum_{R<Q}t_{R}^{Q}(M(R)). 
\end{equation*}
Moreover, when $R=k$ is a field and $V$ is a $kG$-module, then $\overline{FP_{V}}(Q)\cong V[Q]$, where $V[Q]$ is the Brauer construction for modules (\cite{broue_scott} Section $1$). 
\begin{lemma}\label{lemme_brauer}
\begin{itemize}
\item[1.] The functor $M\mapsto M^{Q}$ is left adjoint to the functor \newline
$Ind_{N_{G}(Q)}^{G}Inf_{\overline{N}_{G}(Q)}^{N_{G}(Q)}: Mack_{R}(\overline{N}_{G}(Q))\to Mack_{R}(G)$.
\item[2.] The functor $M\mapsto M^{Q}$ sends projective functors to projective functors. 
\item[3.] Let $H$ be a subgroup of $G$. Then $(Ind_{H}^{G}(M))^{Q}=0$ if $Q$ is not conjugate to a subgroup of $H$.
\end{itemize}
Let $(K,\mathcal{O},k)$ be a $p$-modular system, and $R=\mathcal{O}$ or $k$. 
\begin{itemize}
\item[4.] If $M\in Mack_{R}(G,1)$, then $M^{Q}\in Mack_{R}(\overline{N}_{G}(Q),1)$. 
\item[5.] Let $P$ be a projective Mackey functor of $Mack_{k}(G,1)$ and $L$ be a projective functor of $Mack_{\mathcal{O}}(G,1)$ which lifts $P$. Then $L^{Q}$ is a projective Mackey functor of $Mack_{\mathcal{O}}(G,1)$ which lifts $P^{Q}$.
\item[6.] Let $P$ be a projective Mackey functor of $Mack_{k}(G,1)$. Then $P^{Q}(Q/Q)\cong P(1)[Q]$. 
\item[7.] Let $P$ be an indecomposable projective Mackey functor of $Mack_{R}(G,1)$. Then the vertices of $P$ are the maximal $p$-subgroups $Q$ of $G$ such that $P^{Q}\neq 0$.
\end{itemize}
\end{lemma}
\begin{proof}[Sketch of proof]
\begin{enumerate}
\item Theorem 5.1 of \cite{tw_simple} with a different notation. 
\item Since $M\mapsto M^{Q}$ is left adjoint to an exact functor, it sends projective objects to projective objects. 
\item By successive adjunction: for $L\in Mack_{R}(\overline{N}_{G}(Q))$ and $M\in Mack_{R}(H)$, and using the Mackey formula for Mackey functors, we have:
\begin{align*}
&Hom_{Mack_{R}(\overline{N}_{G}(Q))}((Ind_{H}^{G}M)^{Q},L)\cong Hom_{Mack_{R}(H)}(M,Res^{G}_{H}Ind_{N_{G}(Q)}^{G}Inf_{\overline{N}_{G}(Q)}^{N_{G}(Q)} L)\\
&\cong \bigoplus_{g\in [H\backslash G/N_{G}(Q)]}Hom_{Mack_{R}(H)}\bigg(M,Ind_{H\cap\ ^gN_{G}(Q)}^{H} Iso(g) Res^{N_{G}(Q)}_{N_{G}(Q)\cap H^{g}}Inf_{\overline{N}_{G}(Q)}^{N_{G}(Q)} L\bigg).
\end{align*}
Here we denote by $Iso(g)$ the functor from the category $Mack_{R}(N_{G}(Q)\cap H^{g})$ to the category $Mack_{R}(H\cap\ ^gN_{G}(Q))$ induced by the conjugacy by $g$. The result now follows from the fact that $Res^{N_{G}(Q)}_{N_{G}(Q)\cap H^{g}}Inf_{\overline{N}_{G}(Q)}^{N_{G}(Q)} L=0$ if $Q$ is not a subgroup of $H$.
\item A Mackey functor $M$ is in $Mack_{k}(G,1)$ if and only if there exist a $p$-group $P$ and a Mackey functor $N$ for $P$ such that $M \mid Ind_{P}^{G}(N)$. Moreover, it is not difficult to check that $(Ind_{P}^{G}N)^{Q}$ is a direct sum of Mackey functors induced from subgroups of conjugates of $P$ ($\star$), so $M^{Q}\in Mack_{k}(\overline{N}_{G}(Q),1)$. \newline ($\star$) In order to prove this, one may continue the proof of the point $3$ (or, in a more general setting, look at Lemme $3$ of \cite{bouc_proj}).
\item Let $M$ be a Mackey functor for $\overline{N}_{G}(Q)$. Then using successive adjunctions, we have:
\begin{align*}
Hom_{Mack_{k}(\overline{N}_{G}(Q))}(L^{Q}/\mathfrak{p}(L^{Q}),M)&\cong Hom_{Mack_{k}(\overline{N}_{G}(Q))}(k\otimes_{\mathcal{O}}L^{Q},M)\\
&\cong Hom_{Mack_{\mathcal{O}}(\overline{N}_{G}(Q))}(L^{Q},Hom_{k}(k,M))\\
&\cong Hom_{Mack_{\mathcal{O}}(G)}(L,Ind_{N_{G}(Q)}^{G}Inf_{\overline{N}_{G}(Q)}^{N_{G}(Q)}Hom_{k}(k,M)).
\end{align*}
However, $Ind_{N_{G}(Q)}^{G}Inf_{\overline{N}_{G}(Q)}^{N_{G}(Q)}Hom_{k}(k,M)\cong Hom_{k}(k,Ind_{N_{G}(Q)}^{G}Inf_{\overline{N}_{G}(Q)}^{N_{G}(Q)}(M))$, so
\begin{align*}
Hom_{Mack_{k}(\overline{N}_{G}(Q))}(L^{Q}/\mathfrak{p}(L^{Q}),M)&\cong Hom_{Mack_{\mathcal{O}}(G)}(L,Hom_{k}(k,Ind_{N_{G}(Q)}^{G}Inf_{\overline{N}_{G}(Q)}^{N_{G}(Q)}(M)))\\
&\cong Hom_{Mack_{k}(G)}(L/\mathfrak{p} L, Ind_{N_{G}(Q)}^{G}Inf_{\overline{N}_{G}(Q)}^{N_{G}(Q)}(M))\\
&\cong Hom_{Mack_{k}(\overline{N}_{G}(Q))}((L/\mathfrak{p} L)^Q, M). 
\end{align*}
\item Lemme 5.10 of \cite{bouc_resolution}.
\item This is the first assertion of Theorem $3.2$ of \cite{broue_scott}, using the fact that the vertices of a projective Mackey functor are the vertices of its evaluation at $1$ (Corollary $12.8$ \cite{tw}), and the fact that a $p$-local projective Mackey functor is non zero if and only if its evaluation at $1$ is non zero. 
\end{enumerate}
\end{proof}
Let $F$ be a field and let $G$ be a finite group. By Theorem $8.3$ of \cite{tw_simple}, the set of isomorphism classes of simple Mackey functors for the group $G$ over the field $F$ is in bijection with the set of pairs $(H,V)$ where $H$ runs through a set of representatives of conjugacy classes of subgroups of $G$ and $V$ through a set of isomorphism classes of $F\overline{N}_{G}(H)$-simple modules. We denote by $S_{H,V}$ the simple Mackey functor corresponding to the pair $(H,V)$. Let us recall that:
\begin{equation*}
S_{H,V} = Ind_{N_{G}(H)}^{G} Inf_{\overline{N}_{G}(H)}^{N_{G}(H)} S_{1,V}^{\overline{N}_{G}(H)},
\end{equation*}
where
\begin{equation*}
S_{1,V}^{G}(K)= Im(t_{1}^{K} : V\to V^{K}). 
\end{equation*}
Here $t_{1}^{K}$ is the relative trace map, that is $t_{1}^{K}(v) = \sum_{k\in K} k\cdot v$, for $v\in V$. 
 \newline If $(K,\mathcal{O},k)$ is a $p$-modular system, then a simple Mackey functor $S_{H,V}$ over $k$ is in $Mack_{k}(G,1)$ if and only if $H$ is a $p$-group (see the discussion after Proposition $9.6$ of \cite{tw}). We denote by $P_{H,V}$ the projective cover of $S_{H,V}$. 
\begin{prop}[Decomposition matrix of $\mu_{\mathcal{O}}^{1}(G)$]\label{decomp}
Let $G$ be a finite group, and $(K,\mathcal{O},k)$ be a $p$-modular system which is ``large enough" for the groups $\overline{N}_{G}(L)$ where $L$ runs through the $p$-subgroups of $G$.\newline The decomposition matrix of $\mu_{\mathcal{O}}^{1}(G)$ has rows indexed by the isomorphism classes of indecomposable $p$-permutation $kG$-modules, the columns are indexed by the ordinary irreducible characters of all the $\overline{N}_{G}(L)$ where $L$ runs through the $p$-subgroups of $G$ up to conjugacy.\\ Let $\chi$ be an ordinary character of the group $\overline{N}_{G}(L)$ and let $W$ be an indecomposable $p$-permutation $kG$-module. Then the decomposition number $d_{\chi,W}$ is equal to
\begin{equation*}
d_{\chi,W} = dim_{K} Hom_{K\overline{N}_{L}(Q)}(K\otimes_{\mathcal{O}}\widehat{W[L]},K_\chi),
\end{equation*} 
where $\widehat{W[L]}$ is the (unique) $p$-permutation $\mathcal{O}G$-module which lifts $W[L]$ and $K_{\chi}$ is the simple $KG$-module afforded by the character $\chi$. 
\end{prop}
\begin{proof}
Since $(K,\mathcal{O},k)$ is a splitting system for $\mu_{\mathcal{O}}^{1}(G)$, and since $\mu_{K}^{1}(G)$ is a semi-simple algebra, the Cartan matrix is symmetric (Proposition $1.9.6$ \cite{benson}). The rows of the decomposition matrix are indexed by the projective $\mu_{k}^{1}(G)$-modules, and the columns are indexed by the simples $\mu_{K}^{1}(G)$-modules. Instead of working with the $p$-local Mackey algebras, we use the Mackey functors point of view. Since the indecomposable projective Mackey functors of $Mack_{k}(G,1)$ are in bijection with the indecomposable $p$-permutation $kG$-modules, we can index the rows of this matrix by the (isomorphism classes of) indecomposable $p$-permutation modules. The bijection send a $p$-local projective Mackey functor $P$ to its evaluation $P(1)$ (with Green's notation for the evaluation). Moreover, the (isomorphism classes of) simple Mackey functors of $Mack_{K}(G,1)$ are in bijection with the pairs $(L,V)$, where $L$ runs through the conjugacy classes of $p$-subgroups of $G$ and $V$ runs through the isomorphism classes of simple $KN_{G}(L)/L$-modules. So we can indexed the columns of the matrix by the set of ordinary characters of $N_{G}(L)/L$ when $L$ runs through the conjugacy classes of $p$-subgroups of $G$.
\newline Let $L$ be a $p$-subgroup of $G$ and $\chi$ be an ordinary character of $K\overline{N}_{G}(L)$. The corresponding simple Mackey functor for $G$ over $K$, denoted by $S_{L,K_\chi}$ is isomorphic to $Ind_{N_{G}(L)}^{G}Inf_{\overline{N}_{G}(L)}^{N_{G}(L)}FP_{K_\chi}$ (see Lemma $9.2$ \cite{tw_simple}). Here, $K_{\chi}$ is the simple $K\overline{N}_{G}(L)$-module afforded by the character $\chi$. 
\newline Let $H$ be a $p$-subgroup of $G$ and $V$ be a simple $k\overline{N}_{G}(H)$-module. We denote by $P_{H,V}$ the corresponding projective indecomposable Mackey functor for $G$ over $k$, and by $\widehat{P_{H,V}}$ the indecomposable Mackey functor for $G$ over $\mathcal{O}$ which lifts $P_{H,V}$. 
\newline We denote by $M$ a Mackey functor for $G$ over $\mathcal{O}$ such that $M$ is $\mathcal{O}$-free and $$K\otimes_{\mathcal{O}} M \cong S_{L,K_{\chi}}.$$ Then, the decomposition number indexed by $S_{L,K_{\chi}}$ and $P_{H,V}$ is:
\begin{align*}
d_{S_{L,K_\chi},P_{H,V}}& = dim_{k} Hom_{Mack_{k}(G,1)}(P_{H,V},k\otimes_{\mathcal{O}} M) \\
&= rank_{\mathcal{O}} Hom_{Mack_{\mathcal{O}}(G,1)}(\widehat{P_{H,V}},M)\\
&= dim_{K} Hom_{Mack_{K}(G,1)}(K\otimes_{\mathcal{O}} \widehat{P_{H,V}}, S_{L,K_\chi})\\
&= dim_{K} Hom_{Mack_{K}(G,1)}(K\otimes_{\mathcal{O}} \widehat{P_{H,V}}, Ind_{N_{G}(L)}^{G}Inf_{\overline{N}_{G}(L)}^{N_{G}(L)}FP_{K_\chi})\\
&=dim_{K}Hom_{Mack_{K}(N_{G}(L)/L,1)}(\big(K\otimes_{\mathcal{O}} \widehat{P_{H,V}}\big)^{L},FP_{K_\chi})\\
&= dim_{K} Hom_{KN_{G}(L)/L}(\big(K\otimes_{\mathcal{O}} \widehat{P_{H,V}}\big)^{L}(L/L),K_\chi).
\end{align*}
The two last equalities come from the fact that the two following pairs are pairs of adjoint functors: $ev_{1} \dashv FP_{-}$ and  $(-)^{L} \dashv ind_{N_{G}(L)}^{G}Inf_{\overline{N}_{G}(L)}^{N_{G}(L)}$. 
\newline\indent Moreover, we have $(K\otimes_{\mathcal{O}} \widehat{P_{H,V}})^{L}(L/L)\cong K\otimes_{\mathcal{O}} ((\widehat{P_{H,V}})^{L}(L/L))$. By Lemma \ref{lemme_brauer}, the $\mathcal{O}\overline{N}_{G}(L)$-module $(\widehat{P_{H,V}})^{L}(L/L)$ is the unique (up to isomorphism) $p$-permutation module which lifts $(P_{H,V})^{L}(L/L)\cong P_{H,V}(1)[L]$, so:
\begin{equation*}
d_{S_{L,K_\chi},P_{H,V}} = dim_{K} Hom_{K\overline{N}_{L}(Q)}(K\otimes_{\mathcal{O}}\widehat{W[L]},K_\chi),
\end{equation*}
where $W$ is the indecomposable $p$-permutation $kG$-module $P_{H,V}(1)$. 
\end{proof}
\begin{re}
By Section 4.4 of \cite{bouc_cartan}, the sub-matrix indexed by the ordinary characters of $G$, and the (isomorphism classes of) indecomposable $p$-permutation $kG$-modules is the decomposition matrix of the cohomological Mackey algebra $co\mu_{\mathcal{O}}(G)$. The sub-matrix indexed by the ordinary characters of $G$ and the isomorphism classes of indecomposable projective $kG$-modules is the decomposition matrix of $\mathcal{O}G$. 
\end{re}
\section{Equivalences between blocks of $p$-local Mackey algebras.}
In this section we give some examples of equivalences between blocks of $p$-local Mackey algebras.  We look at the case of $p$-nilpotent groups. We prove that the $p$-local algebra of the principal block of such a group is Morita equivalent to the Mackey algebra of its Sylow $p$-subgroup. But as in \cite{splendid} the case of the non-principal block seems unexpectedly much more difficult. We give an example which proves that in general there is \emph{no} Morita equivalence between the $p$-local Mackey algebra of the block and the Mackey algebra of the defect. Then we look at the case of a finite group with Sylow $p$-subgroups of order $p$. In this case the $p$-local Mackey algebras are symmetric algebras, they are more precisely Brauer tree algebras, so we can use the tools which were developed for Brou\'e's conjecture for blocks with cyclic Sylow $p$-subgroups. 
\subsection{Basic results.}
Since the $p$-local Mackey algebra and the cohomological Mackey algebra share a lot of properties, for example, they have the same number of simple modules in each block and the projective cohomological Mackey functors are the biggest  cohomological quotients of the $p$-local projective Mackey functors, one may ask if an equivalence between blocks of $p$-local Mackey algebras induces in some sense, an equivalence between the corresponding blocks of the cohomological Mackey algebras. Such a result would allow us to use \cite{equiv_coho}. The following example shows that the situation is unfortunately not that simple.
\begin{ex}
Let $R$ be a commutative ring and $G=C_2$ be the (cyclic) group of order $2$. Then a basis of $\mu_{R}(C_2)$ is given by: $t_{C_2}^{C_2}$, $t_{1}^{C_2}r^{C_2}_{1}$, $t_{1}^{C_2}$, $r^{C_2}_{1}$, $t_{1}^{1}$ and $t^{1}_{1}x$ where $x\in C_2$ and $t_{1}^{1}x$ means $t^{1}_{1}c_{1,x}$. Then, there is an automorphism $\phi$ of $\mu_{k}(C_2)$ where $\phi$ is defined on the basis elements by: $\phi(t_{C_2}^{C_2})=t^{1}_{1}$, $\phi(t_{1}^{C_2}r^{C_2}_{1})=t^{1}_{1}+t^{1}_{1}x$, $\phi(t_{1}^{C_2})=r^{C_2}_{1}$, $\phi(r^{C_2}_{1})=t_{1}^{C_2}$, $\phi(t_{1}^{1})=t^{C_2}_{C_2}$ and $\phi(t^{1}_{1}x)=t^{C_2}_{1}r^{C_2}_{1}-t^{C_2}_{C_2}$. This gives an unitary automorphism of $\mu_{R}(C_2)$. By general results of Morita theory, the bimodule $\ _1\mu_{R}(C_2)_{\phi}$ induces a Morita self-equivalence of $\mu_{R}(C_2)$. Using the usual equivalence of categories between $\mu_{R}(C_2)$-$Mod$ and $Mack_{R}(C_2)$ (in the sense of Green), it is not hard to check that this Morita self-equivalence induces an self-equivalence $F$ of $Mack_{R}(C_{2})$ which has the following property. Let $M$ be a Mackey functor for $C_{2}$. Then $F(M)(1) = M(C_{2})$ and $F(M)(C_{2}) = M(1)$. The maps $t_{1}^{C_{2}}$ and $r^{C_{2}}_{1}$ are exchanged. It is now easy to check that this functor $F$ does not preserve the cohomological structure. 
 \newline If $R=k$ is a field of characteristic $2$, then the functor $F$ exchanges the simples $S_{1,k}$ and $S_{C_2,k}$, so it exchanges the indecomposable projective Mackey functors $P_{1,k}$ and $P_{C_{2},k}$. Moreover $P_{1,k}$ is in the full subcategory of $Mack_{R}(G)$ equivalent to $kC_{2}$-$Mod$ and $P_{C_{2},k}$ is not. So we have that $F$ does not restrict to a self-equivalence of he full subcategory of cohomological Mackey functors or of the full subcategory equivalent to $kC_2$-$Mod$.
 \begin{equation*}
 \xymatrix{
 kC_{2}\hbox{-}Mod\ar@{^{(}->}[r]\ar[d]|{\times} & Comack_{k}(C_{2})\ar@{^{(}->}[r]\ar[d]|{\times} & Mack_{k}(C_{2})\ar[d]^{F}\\
 kC_{2}\hbox{-}Mod\ar@{^{(}->}[r] & Comack_{k}(C_{2})\ar@{^{(}->}[r] & Mack_{k}(C_{2})
 }
 \end{equation*}
 \end{ex}
 This example proves the following Proposition:
 \begin{prop}
 Let $G$ and $H$ be two finite groups, and $k$ be a field of characteristic $p>~0$. An equivalence between $Mack_{k}(G,1)$ and $Mack_{k}(H,1)$ \emph{does not} have to send cohomological Mackey functors to cohomological Mackey functors and \emph{does not} have to respect the vertices of the projective indecomposable Mackey functors. In particular it does not have to induce an equivalence between the full-subcategories equivalent to $kG$-$Mod$ and $kH$-$Mod$. 
 \end{prop}
 \begin{re}
 This Proposition does not state that there exist finite groups $G$ and $H$ and blocks $b$ and $c$ of $RG$ and $RH$ such that $Mack_{R}(b)$ and $Mack_{R}(c)$ are equivalent and the category $RGb$-$Mod$ and $RHc$-$Mod$ are not equivalent. We do not have any results in this direction. This Proposition state that we have to be careful when we choose an equivalence between categories of Mackey functors if we want some results for the full subcategories of cohomological Mackey functors or of modules over the group algebra. 
 \end{re}
 \subsection{Principal block of $p$-nilpotent groups.}
\begin{lemma}\label{fitting}
Let $M$ and $M'$ be two projective Mackey functors of $Mack_{k}(G,1)$, and $f: M\to M'$ be a morphism. The morphism $f$ is an isomorphism if and only if the morphism $f(1): M(1)\to M'(1)$ is an isomorphism of $kG$-modules. 
\end{lemma}
\begin{proof}
Lemme $6.2$ in \cite{bouc_resolution}. 
\end{proof}
\begin{theo}\label{nil}
Let $G=N \rtimes P$ be a $p$-nilpotent group, where $P$ is a Sylow $p$-subgroup of $G$. Let $b_{0}$ be the principal block of $kG$. Then $Mack_{k}(b_{0})$ is equivalent to $Mack_{k}(P)$.
\end{theo}
\begin{re}\label{ce}
If $b$ is a nilpotent block, for some $m\in \mathbb{N}$, we have an isomorphism of algebras $kGb\cong Mat(m,kP)$, where $P$ is a defect group of the block $b$. This is not the case for the Mackey algebras. Example if $G=S_{3}$ and $k$ is a field of characteristic $2$. Let $b$ be the principal block of $kS_{3}$, one can check that $dim_{k}(\mu_{k}(C_2))= 6$ and $dim_{k}(\mu_{k}^{1}(b))=56$ (for a proof see Lemme $4.4.2$ \cite{these}).
\end{re}
\begin{proof}
Let us remark that we do not assume the Sylow $p$-subgroups to be abelian.
\newline Recall that the principal block idempotent of $kG$ is $b_{0}=\frac{1}{|N|}\sum_{n\in N}n$, so we have: $kGb_{0}\cong kP$. In consequence, there is a Morita equivalence $kGb_{0}$-$Mod\cong$ $kP$-$Mod$. This equivalence is given by the following pair of adjoint functors $Res^{G}_{P}\dashv b_{0}Ind_{P}^{G}$.
\newline We use the fact that, in our situation, there are three pairs of adjoint functors $res^{G}_{P} \dashv Ind_{P}^{G}$. As we saw, there is one pair for the restriction and induction between the categories of modules over the group algebra. There is an other one between the categories of sets with a group action. Here, we denote by $\epsilon$ and $\eta$ are the unit and co-unit of the usual adjunction $Ind_{P}^{G} \dashv Res^{P}_{G}$ between the categories of finite $G$-sets and finite $P$-sets (see Section $4$ \cite{tw_simple} for an explicit definition of these natural transformations). \\
\newline Finally, at the level of Mackey functors, we also have a restriction and an induction functors. That is: \begin{equation*}Res^{G}_{P}: Mack_{k}(b_{0}) \to Mack_{k}(P),\end{equation*}and \begin{equation*}Ind_{P}^{G}: Mack_{k}(P)\to Mack_{k}(G).\end{equation*} Applying the block idempotent $b_{0}^{\mu}$, we have a functor \begin{equation*}b_{0}^{\mu}Ind_{P}^{G}: Mack_{k}(P)\to Mack_{k}(b_{0}).\end{equation*} 
$1.$ The functor $Res^{G}_{P}$ is both left and right adjoint to the functor $b_{0}^{\mu} Ind_{P}^{G}$, since it is the case for $Ind_{P}^{G}$ and $Res^{G}_{P}$. Recall that the unit and the co-unit of the above adjunction are given by:
\begin{itemize}
\item The unit of the adjunction $Res^{G}_{P} \dashv b_{0}^{\mu} Ind_{P}^{G}$ is the natural transformation $\delta$ defined by: if $M \in Mack_{k}(G)$, then 
\begin{equation*}
\delta_{M}=b_{0}^{\mu} M^{*}(\epsilon) : M\to b_{0}^{\mu}Ind_{P}^{G} Res^{G}_{P}M.
\end{equation*}
The co-unit is the natural transformation $\gamma$ defined by: if $M\in Mack_{k}(P)$, then
\begin{equation*}
\gamma_{M}=b_{0}^{\mu}M^*(\eta) : Res^{G}_{P} b_{0}^{\mu}Ind_{P}^{G}M \to M.
\end{equation*}
\item For the adjunction $b_{0}^{\mu} Ind_{P}^{G}\dashv Res^{G}_{P}$, the unit is the natural transformation $\delta'$ defined by: if $M\in Mack_{k}(P)$, then
\begin{equation*}
\delta'_{M}=M_{*}(\eta) : M\to Res^{G}_{P}b_{0}^{\mu} Ind_{P}^{G}M.
\end{equation*}
The co-unit is the natural transformation $\gamma'$ defined by: if $M\in Mack_{k}(G)$, then
\begin{equation*}
\gamma'_{M} = b_{0}^{\mu}M_{*}(\epsilon): b_{0}^{\mu} Ind_{P}^{G}Res^{G}_{P}M \to M.
\end{equation*}
\end{itemize}
$2.$ Let $M$ be a projective Mackey functor in $Mack_{k}(b_{0})$, and let $N$ be a projective Mackey functor in $Mack_{k}(P)$. Here we use Green's notation for the evaluation. We have a natural isomorphism of $kG$-modules $$(Ind_{P}^{G}N)(1)\cong Ind_{P}^{G}(N(1)),$$
where the first induction functor is the induction of Mackey functors and the second is the induction of $kP$-modules. 
As well, we have an isomorphism of $kP$-modules
\begin{equation*}
Res^{G}_{P}(N)(1) \cong Res^{G}_{P}(N(1)),
\end{equation*}
where the first restriction is a restriction of Mackey functors and the second is the restriction of $kG$-modules. The proof is straightforward but can be found in Proposition $5.2$ \cite{tw}. So the evaluation at $1$ of the units and co-units $\delta_{M},\gamma_{N}$ and $\delta'_{N},\gamma'_{M}$ are the units and co-units of the adjunction $res^{G}_{P} \dashv Ind_{P}^{G}$ between the categories of modules over group algebras. Since the last adjunction is an equivalence of categories, we have that the following composition is equal to $Id_{M(1)}$:
\begin{equation*}
\xymatrix@!{
M(1)\ar[r]^{\delta_{M}(1)} & b_{0}^{\mu} Ind_{P}^{G}Res^{G}_{P}(M)(1) \ar[r]^{\gamma'_{M}(1)} & M(1).
}
\end{equation*}
Moreover, the map $\gamma_{N}(1) \circ \delta'_{N}(1)$ is equal to $Id_{N(1)}$. By Lemma \ref{fitting} the maps $\delta_{M}$ and $\gamma'_{M}$ are two inverse isomorphisms and the maps $\gamma_N$ and $\delta'_{N}$ are two inverse isomorphisms of Mackey functors. 
\newline $3.$ If $M\in Mack_{k}(b_{0})$, let $P_{\bullet}$ be a projective resolution of $M$ in $Mack_{k}(b_{0})$, then we have the following commutative diagram,
\begin{equation*}
\xymatrix{
\cdots \ar[r] & P_{1}\ar[r]\ar[d]^{\delta_{P_1}}& P_{0}\ar[r]\ar[d]_{\delta_{P_{0}}} & M\ar[d]_{\delta_{M}}\ar[r] & 0\\ 
\cdots \ar[r] & b_{0}Ind_{P}^{G}Res^{G}_{P}(P_{1})\ar[r]& b_{0}Ind_{P}^{G}Res^{G}_{P}(P_{0})\ar[r] & b_{0}Ind_{P}^{G}Res^{G}_{P}(M)\ar[r] & 0
}
\end{equation*}
Since the maps $\delta_{P_{i}}$ for $i\geqslant 0$ are isomorphisms, then $\delta_{M}$ is an isomorphism. By the same method, if $N\in Mack_{k}(P)$, then $\gamma_{N}$ is an isomorphism. 
\end{proof}
\begin{re}
One can see that we did not use the fact that the group $G$ is $p$-nilpotent in the proof. Instead we use the fact that the restriction functor is an equivalence between the categories $kGb_{0}$-$Mod$ and $kP$-$Mod$. More generally, let $G$ be a finite group and let $H$ be a subgroup of $G$. Let $b$ be a block of $kG$ and $c$ be a block of $kH$. If the functor $c \cdot Res^{G}_{H} :  kGb$-$Mod \to kHc$-$Mod$ is an equivalence of categories, then $c^{\mu}Res^{G}_{H} : Mack_{k}(b)\to Mack_{k}(c)$ is an equivalence of categories. 
\end{re}
\begin{coro}
There is an isomorphism of algebras $\mu_{k}^{1}(b_{0})\cong kB(X^{2})$ for the $P$-set $X\cong Iso^{P}_{G/N}Def^{G}_{G/N}\Omega_{G}$, and $kB(X^2)$ is the evaluation of the Burnside functor at $X^{2}$ (see example \ref{burnside}).
\end{coro}
\begin{proof}
\noindent By Theorem \ref{nil}, we have an equivalence of categories {$\mu_{k}^{1}(b_{0})$-$Mod\cong$ $\mu_{k}(P)$-$Mod$}, so by Morita Theorem, there is an isomorphism of algebras $\mu_{k}^{1}(b_{0})\cong End_{\mu_{k}(P)}(T)$, where $T$ is the bimodule $Res^{G}_{P}(\mu_{k}^{1}(b_0))$. We will denote by $B_0$ the Mackey functor, in the sense of Dress, which corresponds to the $\mu_{k}^{1}(b_0)$-module $\mu_{k}^{1}(b_0)$ under the usual equivalence of categories. Since the finitely generated projective Mackey functors of $Mack_{k}(P)$ are exactly the Dress constructions $kB_{X}$  of the Burnside functor where $X$ is a finite $P$-set, we have 
\begin{equation*}
Res^{G}_{P}(B_{0}) \cong kB_{X},
\end{equation*}
for some $P$-set $X$. In particular, using Green's notation for the evaluation, we have: $kX\cong Res^{G}_{P}(B_{0})(1)$. But in the equivalence of categories between $\mu_{k}(G)$-$Mod$ and $Mack_{k}(G)$, the evaluation at $1$ correspond to the multiplication by the idempotent $t_{1}^{1}$. So we have: $Res^{G}_{P}(B_{0}(1))=b_{0}t^{1}_{1}\mu_{k}^{1}(G).$\newline A basis of $\mu_{k}^{1}(G)$ is given by $t^{A}_{B}xR^{C}_{B^{x}}$, where $A$ and $C$ are subgroups of $G$,  the elements $x\in [A\backslash G/ C]$ and $B$ is a $p$-subgroup of $A\cap {^{x}C}$ up to $A\cap {^{x}C}$-conjugacy . So a basis of $t^{1}_{1}\mu_{k}^{1}(G)$ is the set of $t^{1}_{1}xr^{H}_{1}$ for $x\in G$ and $H\leqslant G$. Set $\gamma_{H,x} = t^{1}_{1}b_{0}xR^{H}_{1}$. We have that $\gamma_{H,nx}=\gamma_{H,x}$ and $\gamma_{H,xh}=\gamma_{H,x}$ for $x\in G$, $n\in N$ and $h\in H$. The set $\{\gamma_{H,x}\ ;\ H\leqslant G,\ \ x\in G/NH\}$ is a $\mu_{k}(P)$-basis of $t^{1}_{1}\mu_{k}^{1}(b_0)$. The action of $y\in P$ on a element $\gamma_{H,x}$ is given by $y.\gamma_{H,x}=\gamma_{H,yx}$. So, 
\begin{equation*}
kX\cong \bigoplus_{H\leqslant G} Res^{G}_{P}(kG/NH),
\end{equation*}
but for a $p$-group $P$, two permutation modules are isomorphic if and only if the corresponding $P$-sets are isomorphic by \cite{conlon}. Hence
\begin{align*}
X&\cong \sqcup_{H\leqslant G} Res^{G}_{P}(G/NH)\\
&\cong \sqcup_{H\leqslant G}  Res^{G}_{P}Ind^{G}_{G/N}Def^{G}_{G/N}(G/H)\\
&\cong Res^{G}_{P}(Inf^{G}_{G/N}Def^{G}_{G/N}(\Omega_{G}))\\
&\cong Iso^{P}_{G/N} Def^{G}_{G/N}(\Omega_{G}).
\end{align*}
\newline So $Res^{G}_{P}(B_{0})\cong kB_{X}$, where $X=Iso^{P}_{G/N} Def^{G}_{G/N}(\Omega_{G})$, and finally, we have: \begin{equation*}\mu_{k}^{1}(b_0)\cong End_{Mack_{k}(P)}(kB_{X})\cong kB(X^2),\end{equation*} by adjunction property of $kB_{X}$ see (Proposition $3.1.3$ of \cite{bouc_green} in a more general setting). 
\end{proof}
This can be viewed as an analogue of the isomorphism $kGb\cong Mat(n,kP)$ for nilpotent blocks. 
\newline\indent For non principal blocks of $p$-nilpotent groups, or more generally for nilpotent blocks, the situation is much more complicate. The next example shows that, in general, if $b$ is a nilpotent block with defect group $D$, then the Mackey algebra $\mu_{k}^{1}(b)$ is \emph{not} Morita equivalent to the Mackey algebra $\mu_{k}(D)$. 
\begin{ex}\label{ex_sl}
Let $G=SL_{2}(\mathbb{F}_{3})\cong Q_{8}\rtimes C_3$, and $k$ be a field of characteristic 3. The group $G$ is $3$-nilpotent. Let $b$ be the block idempotent such that the block $kGb$ contains the simple $kG$-module $W$ of dimension $2$. Then $kGb$-$Mod\cong$ $kC_{3}$-$Mod$. One can ask if the same happens to the corresponding blocks of the Mackey algebras. But the Cartan matrix of $\mu_{k}(C_3)$ is $\left(\begin{array}{cc}2 & 1 \\1 & 3\end{array}\right)$ and the Cartan matrix of $\mu_{k}^{1}(b)$ is $\left(\begin{array}{cc}3 & 2 \\2 & 3\end{array}\right)$. So there is no Morita equivalence between these two algebras. One can compute these matrices by using Proposition \ref{decomp} or see below.
\newline By Theorem \ref{brauer_tree} they are in fact derived equivalent. In this example it is not difficult to make everything explicit: the projective indecomposable Mackey functors in the block $b^{\mu}$ are in bijection with the indecomposable $p$-permutation modules in the corresponding blocks. These indecomposable $p$-permutation modules are: $Ind_{Q_{8}}^{G}(W)$ and $Ind_{C_6}^{G} Inf_{C_2}^{C_6}k_\epsilon$ where $C_6\cong Z(Q_8)\times C_3$,  and $k_\epsilon$ is the non trivial simple $kC_2$-module. So the projective indecomposable Mackey functors are $P=Ind_{Q_{8}}^{G}FP_{W}$ and $Q=Ind_{C_6}^{G}kB_{\epsilon}$, where $kB_{\epsilon}$ is a direct summand of the Dress construction of the Burnside functor for $C_6$: $kB_{C_6/C_3}$.\newline More precisely, for each $C_6$-sets, if $Y$ is any finite $C_6$-set, then $kB_{C_{6}/C_3}(Y)=kB(Y\times C_{6}/C_{3})$ is a right $kC_{2}$-module, since \begin{equation*}C_6/C_3\cong End_{C_{6}-set}(C_{6}/C_{3})\mapsto End_{C_{6}-set}(Y\times C_{6}/C_{3}).\end{equation*} So $kB_{\epsilon}(Y)=kB(Y\times C_{6}/C_{3})\otimes_{k(C_{6}/C_{3})}k_\epsilon.$ Now this it is not difficult to compute the Cartan matrix of the block.\newline The derived equivalence can be given by a two terms complex: 
\begin{equation*}
\xymatrix{
0\ar[r] & P^{2}\ar[r]^{(\pi,0)}& Q\ar[r] & 0
}
\end{equation*}
where $\pi$ is the morphism of maximal rank between $P$ and $Q$. 
\end{ex}
It should be notice that in this example the Morita equivalence between $kGb$ and $kC_{3}$ is \emph{not} given by a $p$-permutation bimodule. If we want to find an equivalence between these two algebras which respects the subcategories of $p$-permutation modules, we need to replace the bimodule which gives the Morita equivalence by a splendid tilting complex (see \cite{splendid} Section $7.4$).
\subsection{Groups with Sylow $p$-subgroup of order $p$.}
 \begin{lemma}
 Let $b$ be a block of $RG$ with defect group $D$. Then $\mu_{R}^{1}(b)$ is a symmetric algebra if and only if $|D|<p^{2}$. 
 \end{lemma}
 \begin{proof}
By Corollary $4.8$ of \cite{trace_map_mack} then Mackey algebra $\mu_{R}(D)$ is symmetric if and only if $|D|< p^2$. Now, we just have to prove that $\mu_{R}^{1}(b)$ is a symmetric algebra if and only if $\mu_{R}(D)$ is a symmetric algebra. 
Recall that a finite dimensional $R$-algebra $A$ is symmetric if and only if for every finitely generate projective $A$-module $P$ and for every finitely generated $A$-module which is $R$-projective, we have an isomorphism of $R$-modules:
\begin{equation*}
Hom_{A}(P,M)\cong Hom_{R}(Hom_{A}(M,P),R),
\end{equation*}
and this isomorphism is natural in $P$ and $M$ (see Proposition 2.7 \cite{broue_higman}). \\
Suppose that $\mu_{R}^{1}(b)$ is symmetric. Let $L$ be a finitely generated projective $\mu_{R}^{1}(b)$-module i.e. by using the usual equivalence of categories $L$ is a finitely generated projective Mackey functor. Let $M$ be an $R$-projective finitely generated Mackey functors for $D$. Then
\begin{align*}
Hom_{Mack_{R}(D)}(Res^{G}_{D}(L),M)&\cong Hom_{Mack_{R}(G)}(L,b^{\mu}Ind_{D}^{G}(M))\\
&\cong Hom_{R}(Hom_{Mack_{R}(G)}(L,b^{\mu}Ind_{P}^{G}(M)),R)\\
&\cong Hom_{R}(Hom_{Mack_{R}(D)}(Res^{G}_{D}(L),M),R).
\end{align*}
But every finitely projective Mackey functor for the group $D$ is direct summand of the restriction of a finitely generated Mackey functor for $G$. Since the previous isomorphism is natural, it implies an isomorphism between the corresponding direct summands. Conversely, since every projective Mackey functor for $G$ is direct summand of the induction of a projective Mackey functor for $D$, by similar arguments, we have the result. 
 \end{proof}
\begin{theo}\label{brauer_tree}\
Let $G$  and $H$ be two finite groups. Let $b$ be a block of $kG$ with cyclic defect group of order $p$, and $c$ be a block of $kH$ with cyclic defect group of order $p$. Then, \\
$D^{b}(kGb$-$Mod)\cong$ $D^{b}(kHc$-$Mod)$ if and only if $ D^{b}(\mu_{k}^{1}(b)$-$Mod)\cong$ $D^{b}(\mu_{k}^{1}(b')$-$Mod)$.
\end{theo}
\begin{proof}
By Theorem 20.10 of \cite{tw}, in this situation, the blocks of Mackey algebras are Brauer tree algebras. Let $T_{\rm{Mack}}$ be the tree of this algebra. And $T_{\rm{Mod}}$ be the tree of the corresponding block of the group algebra. The tree $T_{\rm{Mod}}$ is isomorphic to a subtree of $T_{\rm{Mack}}$, still denoted by $T_{\rm{Mod}}$. Some properties of the tree $T_{\rm{Mack}}$ are determined by the knowledge of the tree $T_{\rm{Mod}}$. If $e$ is the number of edges of $T_{\rm{Mod}}$, then the number of edges of $T_{\rm{Mack}}$ is $2e$. The exceptional vertex of $T_{\rm{Mack}}$ is the same as the exceptional vertex of the tree $T_{\rm{Mod}}$. Each edge of $T_{\rm{Mack}}$ which is not in $T_{\rm{Mod}}$ is a twig. By general results on derived equivalences for Brauer tree algebras, two Brauer tree algebras with same exceptional multiplicity, over the same field, are derived equivalent if and only if they have the same number of edges.
\end{proof}
Even if the tree $T_{\rm{Mack}}$ seems to be determined by the group algebra $kG$, if two blocks of group algebras are Morita equivalent, it is not always true that the corresponding blocks for the Mackey algebras are Morita equivalent (see Example \ref{ex_sl}). The tree $T_{\rm{Mack}}$ is in fact determined by the corresponding block of $kG$ \emph{and} its Brauer correspondent in $N_{G}(P)$ where $P$ is the defect of the block. 
\begin{re}
In \cite{tw}, all the results are given for blocks of groups with a Sylow $p$-subgroup of order $p$. One can check that these results can be generalized quite easily to general blocks with defect group of order $p$ (see Chapter $6$ of \cite{these} for the proof). 
\end{re}
\begin{prop}\label{morita_p}
Let $G$ and $H$ be two finite groups. Let $b$ (resp. $c$) be a block of $kG$ (resp. $kH$) with a defect $p$-group of order $p$. If $kGb$ and $kHc$ are splendidly Morita equivalent, then there is an equivalence of categories: $\mu_{k}^{1}(e)$-$Mod\cong$ $\mu_{k}^{1}(c)$-$Mod$. 
\end{prop}
\begin{proof}
By Theorem 20.10 of \cite{tw} and Theorem \ref{brauer_tree}, the block algebras $\mu_{k}^{1}(e)$ and $\mu_{k}^{1}(f)$ are derived equivalent Brauer tree algebras. Since two Brauer tree algebras over the same field are Morita equivalent if and only if they have isomorphic trees and same exceptional multiplicity, it is enough too prove that they have the same Cartan matrices. We will prove that the decomposition matrices of $\mu_{\mathcal{O}}^{1}(b)$ and $\mu_{\mathcal{O}}^{1}(c)$ are the same. By Proposition \ref{decomp}, the decomposition matrices of $\mu_{\mathcal{O}}^{1}(b)$ can be computed from the knowledge of the $p$-blocks $\mathcal{O}Gb$ and the Brauer correspondent of $b$ in $\mathcal{O}\overline{N}_{G}(C)$. 
\newline\indent Suppose that there are $e$ simple $kG$-modules in the block $b$ of $kG$, then there are $e+1$ simple $KG$-modules in this block, one of this simple module may be exceptional. The number of simple $k\overline{N}_{G}(C)$-modules $W$ in a block $b'$ which is the Brauer correspondent of $\overline{b}$ is $e$. Since $\overline{N}_{G}(C)$ is a $p'$-group, each simple $k\overline{N}_{G}(C)$-module gives rise to a unique simple $K\overline{N}_{G}(C)$-module. Thus the decomposition matrix of $\mu_{\mathcal{O}}^{1}(b)$ is the following block matrix with $2e+1$ columns and $2e$ rows: 
\[
\left(
\begin{array}{ccc}
\begin{array}{ccc}
&\\
 &D(co\mu_{\mathcal{O}}(b))\\
& \\
\end{array}
\begin{array}{ccc}
& 0_{e\times e}\\
& \\
& Id_{e\times e}
\end{array} 
\end{array}
\right)
\]
Where $D(co\mu_{\mathcal{O}}(b))$ is the decomposition matrix of $co\mu_{\mathcal{O}}(b)$. 
\newline\indent So if two blocks  $kGb$ and $kHc$, with cyclic defect group of order $p$ are splendidly Morita equivalent, the blocks $\mathcal{O}Gb$ and $\mathcal{O}Hc$ are splendidly Morita equivalent by Section 5 of \cite{splendid}. By Proposition $4.4$ of \cite{equiv_coho}, the blocks of the cohomological Mackey algebras $co\mu_{\mathcal{O}}(b)$ and $co\mu_{\mathcal{O}}(c)$ are Morita equivalent, so the Cartan matrices of $\mu_{k}^{1}(b)$ and $\mu_{k}^{1}(c)$ are the same. 
\end{proof}
\begin{coro}
Let $G$ and $H$ be two finite groups. Let $b$ be a block of $\mathcal{O}G$ and let $c$ be a block of $\mathcal{O}H$. Let us suppose that the blocks $b$ and $c$ have a defect group of order $p$. Then
\begin{enumerate}
\item $D^{b}(\mu_{\mathcal{O}}^{1}(b))\cong D^{b}(\mu_{\mathcal{O}}^{1}(c))$ if and only if $D^{b}(\mathcal{O}Gb) \cong D^{b}(\mathcal{O}Hc)$.
\item If $\mathcal{O}Gb$ and $\mathcal{O}Hc$ are splendidly Morita equivalent, then $\mu_{\mathcal{O}}^{1}(b)$-$Mod \cong \mu_{\mathcal{O}}^{1}(c)$-$Mod$. 
\end{enumerate}
\end{coro}
\begin{proof}
For this proof, we suppose that the reader is familiar with the notion of \emph{Green orders}. For the definition and properties of Green orders see \cite{derived_zimmermann} and \cite{derived_order}. 
\newline The Mackey algebra $\mu_{k}^{1}(b)$ over $k$ of a block with defect of order $p$ is a Brauer tree algebra, so there is a Green walk on this tree. One can lifts this Green walk for  $\mu_{\mathcal{O}}^{1}(b)$ exactly as Green did in \cite{green_walk}. This shows that the Mackey algebra over $\mathcal{O}$ is a Green order in the sense of \cite{green_order}. K\"onig and Zimmermann in \cite{derived_order} proved that two \emph{generic} Green orders with trees having same number of vertices and same exceptional vertex plus some local properties (the orders attached to the vertices have to coincide) are derived equivalent. In our situation, that is: group algebras or $p$-local Mackey algebras over a $p$-modular system which is `large enough', one can prove that, the order attached to a non exceptional vertex is $\mathcal{O}$ and the order attached to the exceptional vertex depends only on the exceptional multiplicity (for more details see Lemme $6.3.24$ of \cite{these}). The first part of the theorem follows. The second part follows from the fact that two generic Green orders are Morita equivalent if and only if they have isomorphic trees, with same exceptional multiplicity and the local datas of the two trees coincide. Since the exceptional datas of the trees of $\mu_{\mathcal{O}}^{1}(b)$ and $\mu_{\mathcal{O}}^{1}(c)$ coincide, the result follows from Proposition \ref{morita_p}
\end{proof}
\paragraph{Acknowledgements
}This work is part of my PhD at Universit\'e de Picardie Jules Verne. I would like to thank Serge Bouc, my PhD advisor for many helpful conversations. I would like to thank Alexander Zimmermann for his help with the notion of Green orders. I thank the foundation FEDER and the CNRS for their financial support. I finally thank ECOS and CONACYT for the financial support in the project M10M01.
\par\noindent
{Baptiste Rognerud\\
EPFL / SB / MATHGEOM / CTG\\
Station 8\\
CH-1015 Lausanne\\
Switzerland\\
e-mail: baptiste.rognerud@epfl.ch}

\begin{thebibliography}{10}

\bibitem{benson}
D.~Benson.
\newblock {\em Representations and cohomology : Basic representation theory of
  finite groups and associative algebras}, volume~I.
\newblock {C}ambridge university press, 1995.

\bibitem{bouc_proj}
S.~Bouc.
\newblock Projecteurs dans l'anneau de {B}urnside, projecteurs dans l'anneau de
  {G}reen, et modules de {S}teinberg g\'en\'eralis\'es.
\newblock {\em J. Algebra}, 139(2):395--445, 1991.

\bibitem{bouc_green}
S.~Bouc.
\newblock {\em Green Functors and $G$-sets}, volume 1671 of {\em Lecture Notes
  in Mathematics}.
\newblock Springer, 1997.

\bibitem{bouc_resolution}
S.~Bouc.
\newblock R{\'e}solutions de foncteurs de {M}ackey.
\newblock {\em Proc. Sympos. Pure Math., 63, Amer. Math. Soc.}, pages 31--83,
  1998.

\bibitem{bouc_blocks}
S.~Bouc.
\newblock The $p$-blocks of the {M}ackey algebra.
\newblock {\em Algebras and Representation Theory}, 6:515--543, 2003.

\bibitem{bouc_cartan}
S.~Bouc.
\newblock On the {C}artan matrix of {M}ackey algebras.
\newblock {\em Trans. Amer. Math. Soc.}, 363(8):4383--4399, 2013.

\bibitem{broue_scott}
M.~Brou{\'e}.
\newblock On {S}cott modules and $p$-permutation modules: an approach through
  the {B}rauer morphism.
\newblock {\em Proceedings of the American Mathematical Society},
  93(3):401--408, March 1985.

\bibitem{broue_higman}
M.~Brou{\'e}.
\newblock Higman´s criterion revisited.
\newblock {\em The Michigan Mathematical Journal}, 58(1):125--179, 05 2009.

\bibitem{conlon}
S.~B. Conlon.
\newblock Decompositions induced from the {B}urnside algebra.
\newblock {\em J. Algebra}, 10:102--122, 1968.

\bibitem{dress}
A.~Dress.
\newblock {\em Contributions to the theory of induced representations}, volume
  342 of {\em Lecture Notes in Mathematics}.
\newblock Springer Berlin / Heidelberg, 1973.

\bibitem{green}
J.~A. Green.
\newblock {A}xiomatic representation theory for finite groups.
\newblock {\em Journal of Pure and Applied Algebra}, 1(1):41--77, 1971.

\bibitem{green_walk}
J.~A. Green.
\newblock Walking around the {B}rauer tree.
\newblock {\em J. Austral. Math. Soc.}, 17:197--213, 1974.

\bibitem{derived_order}
S.~K{\"o}nig and A.~Zimmermann.
\newblock Tilting selfinjective algebras and {G}orenstein orders.
\newblock {\em Quart. J. Math. Oxford Ser. (2)}, 1997.

\bibitem{derived_zimmermann}
S.~{K}{\"o}nig and A.~{Z}immermann.
\newblock {\em {D}erived {E}quivalences for {G}roup {R}ings}.
\newblock Lecture Notes in Mathematics. Springer Berlin Heidelberg, 1998.

\bibitem{nilpotent_revisited}
B.~K{\"u}lshammer.
\newblock {N}ilpotent blocks revisited.
\newblock In {\em Groups, rings and group rings}, volume 248 of {\em Lect.
  Notes Pure Appl. Math.}, pages 263--274. Chapman \& Hall/CRC, Boca Raton, FL,
  2006.

\bibitem{nilpotent_puig}
L.~Puig.
\newblock Nilpotent blocks and their source algebras.
\newblock {\em Invent. Math.}, 93(1):77--116, 1988.

\bibitem{splendid}
J.~Rickard.
\newblock Splendid equivalences : Derived categories and permutation modules.
\newblock {\em Proc. London Math. Soc.}, pages 331--358, 1996.

\bibitem{green_order}
K.~W. Roggenkamp.
\newblock Blocks of cyclic defect and {G}reen orders.
\newblock {\em Comm. Algebra}, 20(6):1715--1734, 1992.

\bibitem{these}
B.~Rognerud.
\newblock {\em Equivalences de blocs d'alg{\`e}bres de Mackey.}
\newblock PhD thesis, Universit\'e de Picardie Jules Verne, {D}ecember 2013.

\bibitem{trace_map_mack}
B.~Rognerud.
\newblock Trace {M}aps for {M}ackey {A}lgebras.
\newblock preprint, december 2013.

\bibitem{equiv_coho}
B.~Rognerud.
\newblock Equivalences between blocks of cohomological mackey algebras.
\newblock preprint, 2014.

\bibitem{tw}
J.~Th{\'e}venaz and P.~Webb.
\newblock The structure of {M}ackey functors.
\newblock {\em Transactions of the American {M}athematical {S}ociety},
  347(6):1865--1961, 1995.

\bibitem{tw_simple}
J.~Th{\'e}venaz and P.~J. Webb.
\newblock {S}imple {M}ackey functors.
\newblock In {\em Proceedings of the {S}econd {I}nternational {G}roup {T}heory
  {C}onference ({B}ressanone, 1989)}, 1990.

\end{thebibliography}
\end{document}